\documentclass[12pt]{amsart}

\usepackage{amssymb}
\usepackage{verbatim}

\newtheorem{thm}{Theorem}[section]

\newtheorem{lem}[thm]{Lemma}
\newtheorem{cor}[thm]{Corollary}

\theoremstyle{definition}

\theoremstyle{remark}

\numberwithin{equation}{section}

\newcommand{\R}{\mathbf{R}}  
\newcommand{\N}{\mathbf{N}}
\newcommand{\C}{\mathbf{C}}
\newcommand{\Z}{\mathbf{Z}}

\newcommand{\Mod}[1]{\ (\textup{mod}\ #1)}

\providecommand{\id}{\operatorname{1}}

\begin{document}


\title[Non-vanishing of automorphic $L$-functions]{Non-vanishing of automorphic $L$-functions of prime power level}


\author{Olga  Balkanova}
\address{Institute for Applied Mathematics of Russian Academy of Sciences,  Khabarovsk, Russia}
\email{olgabalkanova@gmail.com}

\author{Dmitry  Frolenkov}
\address{Steklov Mathematical Institute of Russian Academy of Sciences, Moscow, Russia}
\email{frolenkov@mi.ras.ru}
\begin{abstract}
 Iwaniec and Sarnak  showed that at the minimum $25\%$ of $L$-values  associated to holomorphic newforms of fixed even integral weight and level $N \rightarrow \infty$ do not vanish at the critical point when $N$ is square-free and $\phi(N)\sim N$. In this paper we extend the given result to the case of prime power level $N=p^{\nu}$, $\nu\geq 2$.
\end{abstract}


\maketitle

\tableofcontents

\section{Introduction}
Central values of $L$-functions associated to different families is an important subject in analytic number theory with numerous applications.
In this paper we consider the family $H^{*}_{2k}(N)$ of primitive newforms of level $N$ and weight $2k$ equipped with the Petersson inner product
\begin{equation} \label{Pet}
\langle f,g\rangle_N:=\int_{F_0(N)}f(z)\overline{g(z)}y^{k}\frac{dxdy}{y^2}.
\end{equation}
Here $F_0(N)$ is a fundamental domain of the action of the Hecke congruence
subgroup $\Gamma_0(N)$ on the upper-half plane $$\mathbb{H}=\{z \in \C: \Im{z}>0\}.$$

 The main advantage of considering primitive forms $f \in H_{2k}^{*}(N)$ is that the  coefficients appearing in the Fourier expansion
\begin{equation}
f(z)=\sum_{n\geq 1}\lambda_f(n)n^{(2k-1)/2}e(nz)
\end{equation}
 are multiplicative
\begin{equation}\label{multiplicity}
\lambda_f(m)\lambda_f(n)=\sum_{\substack{d|(m,n)\\(d,N)=1 }}\lambda_{f}\left( \frac{mn}{d^2}\right).
\end{equation}

For  each $f \in H^{*}_{2k}(N)$ we associate an $L$-function
\begin{equation}L_f(s)=\sum_{n \geq 1}\frac{\lambda_f(n)}{n^s}, \quad \Re{s}>1
\end{equation}
such  that the completed $L$-function
\begin{equation}
\Lambda_f(s)=\left(\frac{\sqrt{N}}{2 \pi}\right)^s\Gamma\left(s+\frac{2k-1}{2}\right)L_f(s)
\end{equation} satisfies the functional equation
\begin{equation} \label{eq: functionalE}
\Lambda_f(s)=\epsilon_f\Lambda_f(1-s), \quad \epsilon_f=\pm1
\end{equation}
and can be analytically continued on the whole complex plane.

The point $s=1/2$ is the symmetry point of functional equation \eqref{eq: functionalE}.
For $N$ large  enough the harmonic proportion of non-vanishing is defined by
 \begin{equation}\label{pn}
PN:=\sum_{\substack{f \in H_{2k}^{*}(N)\\ L_f(1/2)\neq 0}}^{h}1
\geq c-\epsilon, \quad \epsilon>0.
\end{equation}
Subscript $h$ indicates that the average is taken with the weight\footnote{ A method of removing the harmonic weight is described in \cite{KM}.}
$$\frac{\Gamma(2k-1)}{(4\pi)^{2k-1} \langle f,f\rangle_N}.$$

Brumer's conjecture predicts that $c=1/2$ as $N \rightarrow \infty$.
This is still an open problem but there are some remarkable results towards its solution:
\begin{itemize}
 \item
Duke \cite{Duk}: $c=A/\log^2{N}$, $N$ is prime, $A$ is a positive constant;

\item
Vanderkam \cite{Van}: $c=1/48$, $N$ is prime;
\item
Kowalski-Michel \cite{KM}: $c=1/6$, $N$ is prime;

\item
Iwaniec-Sarnak \cite{IS}: $c=1/4$, $N$ is square-free and $\phi(N)\sim N$;
\item
Rouymi \cite{R2}: c$=\frac{p-1}{p}1/6$, $N=p^{\nu}$, $p$ is a fixed prime, $\nu \rightarrow \infty$.

\end{itemize}
The value  $c=1/4$ is a natural barrier. Iwaniec and Sarnak \cite{IS} showed that if inequality \eqref{pn} holds for some $c>1/4$ with an additional lower bound $$L(1/2,f)\geq 1/(\log{N})^2,$$ there are no Landau-Siegel zeros
for Dirichlet $L$-functions of real primitive characters.

The aim of the present paper is to improve the results of Rouymi by proving a lower bound on the proportion of non-vanishing central values of $L$-function of level $N=p^{\nu},$ $\nu \geq 2$ which corresponds to the one obtained by Iwaniec and Sarnak \cite{IS} in case of square-free level. Our main result is the following.
\begin{thm}\label{thm:main}
Let $N=p^{\nu}$, $p$ is a fixed prime. For any $\epsilon>0$
 \begin{equation}\sum_{\substack{f \in H_{2k}^{*}(N)\\ L_f(1/2)\neq 0}}^{h}1 \geq \frac{\phi(N)}{N}\frac{1}{4}-\epsilon, \quad \text{as } \nu \rightarrow \infty. 
\end{equation}
Let $N=p^{\nu}$, $\nu \geq 2$ is fixed. For any $\epsilon>0$
\begin{equation}\label{secondcase}\sum_{\substack{f \in H_{2k}^{*}(N)\\ L_f(1/2)\neq 0}}^{h}1 \geq \frac{1}{4}-\epsilon, \quad
\text{as } p \rightarrow \infty  \text{ over primes}.
\end{equation}
\end{thm}

The proof is based on  asymptotic evaluation of twisted moments
\begin{equation}
M_1(l,u,v)=\sum_{f \in H_{2k}^{*}(N)}^{h}\lambda_f(l)L_f(1/2+u+v),
\end{equation}
\begin{equation}
M_2(l,u,v)=\sum_{f \in H_{2k}^{*}(N)}^{h}\lambda_f(l)L_f(1/2+u+v)L_f(1/2+u-v).
\end{equation}

The previously known results for prime power level are due to Rouymi.
\begin{lem}\label{lem:R}(lemmas 4 and 5 of \cite{R2})
Let  $k\geq 1$, $N=p^{\nu}$, $p$ is prime,  $\nu \geq 3$, $(l,p)=1$. Then for all $1\leq l \leq N$ and $\epsilon>0$  we have
\begin{equation}\label{m1rouymi}
M_1(l,0,0)=\frac{1}{\sqrt{l}}\frac{\phi(N)}{N}+O_{k,p}\left(\frac{\sqrt{l}\log{^2}(2l)}{N^{1/4-\epsilon}} \right);
\end{equation}
\begin{multline}\label{m2rouymi}
M_2(l,0,0)=\frac{\tau(l)}{\sqrt{l}}\left( \frac{\phi(N)}{N}\right)^2\times \\ \times
\left( \log{\left( \frac{N}{4\pi^2l}\right)}+2 \left( \frac{\log{p}}{p-1}+\frac{\Gamma'}{\Gamma}(k/2)+\gamma\right)\right)+O_{k,p}\left(\sqrt{l}\frac{(\log{N})^4}{\sqrt{N}}\right).
\end{multline}
\end{lem}
Using lemma \ref{lem:R} and the technique of mollification Rouymi obtained a lower bound on the proportion of non-vanishing.
\begin{thm}\label{Royumi theorem}(theorem $1$ of \cite{R2})
Let $k \geq 1$ and $p$ be a fixed prime number. For all $\epsilon>0$, $\Omega<1/4$ there is a constant $\nu_0=\nu_0(p,k,\epsilon,\Omega)$ such that for $\nu\geq \nu_0$ and $N=p^{\nu}$ we have
\begin{equation}\label{proportion}
PN=\sum_{\substack{f \in H_{2k}^{*}(N)\\ L_f(1/2)\neq 0}}^{h}1\geq  \frac{p-1}{p}\frac{\Omega}{1+2\Omega}-\epsilon.
\end{equation}
\end{thm}
The parameter $L:=N^{\Omega}$ is called the length of mollifier and $\Omega$ is called the logarithmic length of mollifier.  As $v \rightarrow \infty $ one can take $\Omega=1/4-\epsilon$ in \eqref{proportion} so that $$PN \geq \frac{p-1}{p}\frac{1}{6}-\epsilon.$$
In his papers Rouymi didn't consider the opposite case: $N=p^{\nu}$, $\nu$ is fixed and $p\rightarrow \infty$ over primes.
However, using lemma \ref{tracebound}, it is possible to obtain asymptotic formulas
for \eqref{m1rouymi} and \eqref{m2rouymi} that are uniform in both parameters $p$ and $\nu$.
\begin{lem}\label{tracebound} (theorem $2.2$ of \cite{R0}, page $23$)
Let $k\geq 1$, $N=p^{\nu}$, $\nu \geq 2$, $(p,mn)=1$. Then
\begin{equation*}
\sum_{f \in H^{*}_{2k}(N)}\lambda_f(n)\lambda_f(n)=\phi_{\nu}(N)\delta_{m,n}+
O\left( \frac{\sqrt{mnp}\{\log{(2(m,n))}\}^2}{k^{4/3}N^{3/2}}\right),
\end{equation*}
where
\begin{equation}
\phi_{\nu}(N)=
\left\{
              \begin{array}{ll}
                1-p^{-1}, & \hbox{if $N=p^{\nu},\,\nu\geq3$;} \\
                1-(p-p^{-1})^{-1}, & \hbox{if $N=p^{\nu},\,\nu=2$.}
              \end{array}
\right.
\end{equation}
\end{lem}
This gives an extra factor $\sqrt{p}$ in the error terms for the first and second moments.
Therefore, one can take  $\Omega<1/4-1/(4\nu)$ for the second moment and $\Omega < 1/4-1/(2\nu)$ for the first moment.
This means that the length of mollifier is at most $N^{1/4-1/(2\nu)}$. As a consequence, theorem \ref{Royumi theorem} cannot be extended to the case $N=p^2.$ If $\nu>2$ we, for example, have
\begin{align*}
\nu=3, p \rightarrow \infty&& \Omega=1/12-\epsilon&& PN\geq 1/14-\epsilon;\\
\nu=4, p \rightarrow \infty&& \Omega=1/8-\epsilon&& PN\geq 1/10-\epsilon.
\end{align*}

To make the proportion of non-vanishing independent of $p$ and prove \eqref{secondcase}, we obtain estimates on the error terms for the first and second moments when $l>N$.
Then theorem \ref{thm:main} is proved by extending the length of mollifier up to $N$ for the first moment and up to $N^{1/2}$ for the second moment. Therefore, any $\Omega<1/2$ is admissible.

Asymptotic formula for the second moment, which is the hardest part of our arguments, is derived in \cite{BalFrol}. The proof is based on techniques developed by Kuznetsov and Bykovskii, see \cite{Byk, BF, BF2, Kuz}. Here we give a simplified version of the main theorem.
\begin{thm}\label{thm:main2}(theorem $1.1$ of \cite{BalFrol})
Let $k\geq 1 $, $p$ be a prime, $t \in \R$, $T:=3+|t|$, $N=p^{\nu}$  and $\nu \geq 2$.
If $p|l$ then $M_2(l,0,it)=0$. For $(l,p)=1$ we have
\begin{multline}\label{eq:mt}
M_2(l,0,it)=\frac{\phi(N)\phi_{\nu}(N)}{N}\frac{\tau_{it}(l)}{l^{1/2}}\times\\ \times \Biggl(
\log{N}+2\gamma-2\log{2\pi}+2\frac{\log{p}}{p-1}+\frac{\Gamma'}{\Gamma}(k+it)+\frac{\Gamma'}{\Gamma}(k-it)\Biggr)-\\
-\frac{\phi(N)\phi_{\nu}(N)}{N}\frac{1}{l^{1/2-it}}\sum_{d|l}\sum_{r|d}\mu(d/r)\tau_{it}(r)r^{-it}\log{r}+
O_{\epsilon}\left( (lNT)^{\epsilon}\frac{\sqrt{l}T}{N}\right),
\end{multline}
where
\begin{equation}
\tau_{it}(r)=\sum_{ab=r}\left(\frac{a}{b}\right)^{it}.
\end{equation}
\end{thm}

The second ingredient of our proof is the asymptotic evaluation of $M_1(l,0,it)$.
The first moment over the full  basis of cusp forms (or over the space of primitive forms of prime level and small weight) has been studied intensively in different aspects. See, for example,   \cite{Ak}, \cite{Bet},  \cite{Duk}, \cite{Ell},  \cite{JK},  and \cite{Kam}.
In case of prime power level  and weight $2$ the best known error term for  $M_1(l,0,it)$ with respect to parameters $N$, $T$ and $l$   is obtained in \cite{Bet}.
\begin{thm}\label{thm:Bet} (particular case of theorem $1.1$ of \cite{Bet})
Let $2k=2$,  $t \in \R$, $T=1+|t|$, $N=p^{\nu}$, $p$ is a prime such and $\nu \geq 2$.
If $p|l$ then $M_1(l,0,it)=0$.
For $(l,p)=1$ uniformly in $N$, $l$, $T$, $p$
 $$M_1(l,0,it)= \frac{\phi_{\nu}(N)}{l^{1/2+it}}+O_{\epsilon}\left(\sqrt{lT}N^{-1+\epsilon}\right).$$
\end{thm}

In sections \ref{section4} and \ref{section5} we prove an asymptotic formula for the first moment when $k>1$. Our formula is uniform in all considered parameters and allows extending the logarithmic length of mollifier $\Omega$ up to $1$ in both weight and level aspects.
\begin{thm}\label{thm:BF}
Let $2k \geq 4 $, $t \in \R$, $T=1+|t|$, $N=p^{\nu}$, $p$ is a prime  and $\nu \geq 2$.
If $p|l$ then $M_1(l,0,it)=0$. For $(l,p)=1$
uniformly in $N$, $l$, $T$, $p$ and $k$
\begin{equation*}
M_1(l,0,it)=\frac{\phi_{\nu}(N)}{l^{1/2+it}}
+O\left( V_{N}(l)+\frac{1}{p}V_{N/p}(l)+ \delta_{N,p^2}\frac{1}{\sqrt{l}p^2}\right),
\end{equation*}
where
\begin{equation*}
 V_N(l)\ll \begin{cases}
                 \frac{1}{\sqrt{lT}}\left(2\pi e \frac{lT}{Nk}\right)^k, & l< \frac{1}{4\pi e}\frac{Nk}{T}\\
                \frac{l^{1/2}}{N}(lT)^{\epsilon}\max{\left( \frac{\sqrt{T}}{k},
             \frac{1}{\sqrt{k}}\right)}, & l\geq \frac{1}{4\pi e}\frac{Nk}{T}.
            \end{cases}
\end{equation*}
\end{thm}
Combining the last two estimates in one, we have
\begin{equation*}
V_N(l)+\frac{1}{p}V_{N/p}(l) + \delta_{N,p^2}\frac{1}{\sqrt{l}p^2} \ll_{k,\epsilon} \frac{l^{1/2+\epsilon}T^{1/2+\epsilon}}{N}.
\end{equation*}
This is consistent with the error term in theorem \ref{thm:Bet} and this improves the results of  \cite{Ich},   \cite{R}  and \cite{R2}. For $l=1$ the improvement is even more significant: we obtain the error term $O_{k,\epsilon} (N^{-k})$ instead of $O_{k,p}\left(N^{-1/2}\right)$ proved in \cite{R}.

Our method is different from the techniques applied in the prior literature. In particular, it is not based on the approximate functional equation. Instead we compute the explicit formula for the shifted first moment and then continue it analytically to the critical point. An advantage of such approach is a better understanding of the structure of error terms. The case $N=p^2$  is particularly interesting since  asymptotic formula has an additional error term coming from a pole of the Lerch zeta function. We study this case separately in section \ref{section4}.

The paper is organized as follows. In section \ref{section2} we remind the reader of some background information. Section \ref{section3} is devoted to studying a special function which appears as a part of an error term. Asymptotics of the first moment is computed in sections \ref{section4} and \ref{section5}.  In the last section  we prove the main theorem using the technique of mollification.
\section{Background information}\label{section2}
 We follow the notations of \cite{BalFrol} and let
\begin{equation}
\id_{(l,p)=1}=\begin{cases}
1& (l,p)=1,\\
0& p|l
\end{cases}
\end{equation}
and
\begin{equation}
\delta_q(n)=\begin{cases}
1& n\equiv 0\Mod{q},
\\ 0& \text{ otherwise.}
\end{cases}
\end{equation}

The Bessel function
\begin{equation}
J_s(z)=\sum_{n=0}^{\infty}\frac{(-1)^n}{\Gamma(n+1)\Gamma(n+1+s)}\left(\frac{z}{2}\right)^{s+2n}
\end{equation}
satisfies the Mellin-Barnes representation
\begin{equation}\label{Barnes}
J_{2\lambda-1}(y)=\frac{1}{4\pi i}\int\limits_{\Re s=\Delta}
\frac{\Gamma(\lambda-1/2+s/2)}{\Gamma(\lambda+1/2-s/2)}\left(\frac{y}{2}\right)^{-s}ds
\end{equation}
for $1-2\Re{\lambda}<\Delta<0$ and positive real  $y$.
Let
\begin{equation}
e(x):=\exp(2\pi i x).
\end{equation}
The Lerch zeta function with parameters $\alpha, \beta \in \R$ is defined by
\begin{equation}
\zeta(\alpha,\beta;s)=\sum_{n+\alpha>0}\frac{e(n\beta)}{(n+\alpha)^s}
\end{equation}
for $\Re{s}>1$.
This a periodic function on $\beta$ with a period one, $\zeta(\alpha,\beta;s)$ can be analytically continued
on the whole complex plane except the point $s=1$ for $\beta \in \Z$, where it has a simple pole with residue $1$.
One has the functional equation
\begin{equation}\label{Lerch.func.equat}
\xi(0,\alpha;1-s)=\frac{\Gamma(s)}{(2\pi)^s}
\left(
e\left(\frac{s}{4}\right)\xi(\alpha,0;s)+e\left(-\frac{s}{4}\right)\xi(-\alpha,0;s)
\right).
\end{equation}
The classical Kloosterman sum is defined by
\begin{equation}
Kl(m,n;c)=\sum_{\substack{x \Mod{c}\\ (x,c)=1}}e\left(\frac{mx+nx^{*}}{c}\right),\quad
xx^{*}\equiv 1\Mod{c}.
\end{equation}
\begin{lem}(\cite{Royer}, lemma A.12)\label{Royer}
Let $m,n,c$ be three strictly positive integers and $p$ be a prime
number. Suppose $p^2|c$, $p|m$ and $p\not{|} n$. Then $Kl(m,n;c)=0$.
\end{lem}

The key tool of our computations is the Petersson trace formula.
\begin{thm}\label{Petersson}(Petersson's trace formula)
For $m,n \geq 1$ we have
\begin{multline} \label{eq:Petersson}
\Delta_{2k,N}(m,n)=\sum_{f\in H_{2k} (N)}^{h}
\lambda_f(m)\lambda_f(n)=\\=\delta(m,n)+ 2\pi i^{-2k}\sum_{c\equiv 0 \Mod{N}}
\frac{Kl(m,n;c)}{c} J_{2k-1} ( \frac{4\pi\sqrt{mn}}{c} ),
\end{multline}
where $H_{2k} (N)$ is a basis of the space of cusp forms of level $N$ and weight $2k$.
\end{thm}

\begin{thm}\label{PetRou}(remark $4$ of \cite{R})
Let  $N=p^{\nu}$ with prime $p$ and $\nu \geq 2$. Then
\begin{multline}\label{eq:rtrace}
\Delta_{2k,N}^{*}(m,n):=\sum_{f\in H_{2k}^{*}(N)}^{h}
\lambda_f(m)\lambda_f(n)=\\=\begin{cases}
\Delta_{2k,N}(m,n)-\frac{\Delta_{2k,N/p}(m,n)}{p-p^{-1}}& \text{if $\nu=2$ and $(N,mn)=1$},\\
\Delta_{2k,N}(m,n)-\frac{\Delta_{2k,N/p}(m,n)}{p}& \text{if $\nu \geq 3$ and $(N,mn)=1$},\\
0& \text{if } (N,mn)=1.
\end{cases}
\end{multline}
 \end{thm}

\section{The error term}\label{section3}
Let $k>1+\Re{u}\geq1$,  $\Re{v}=0$, $\Im{v}=t$ with $t \in \R$ and $T:=1+|t|$.
The main object of this section is the function
\begin{multline}\label{int:I}
 I_{\pm}(u,v,k,x)\\
=e\left( \pm \frac{1}{8}\mp \frac{k}{4} \right)x^{1/2-k}\frac{\Gamma(k-v-u)}{\Gamma(2k)}
{}_{1}F_{1}\left(k-v-u,2k;-\frac{e(\mp 1/4) }{x}\right)
\\=\frac{1}{2\pi i}\int_{\Delta/2} \frac{\Gamma(k-1/2+s)}{\Gamma(k+1/2-s)}\Gamma(1/2-u-v-s)
x^{s}e\left( \pm \frac{s}{4}\right)ds
\end{multline}
with $1-2k<\Delta<-1-2\Re u$. This integral appears in the error term of the asymptotic formula for the first moment.
We estimate $I_{\pm}(0,v,k,x)$ distinguishing two cases of large and small argument $x$.

\subsection{The case of large $x$}
\begin{lem}\label{lem:Ilargex}
 For $x>Te/k$ we have
\begin{equation}
 I_{\pm}(0,v,k,x)\ll e^{-\pi T/2}\left(\frac{x}{T} \right)^{1/2}\left(\frac{eT}{xk} \right)^k.
\end{equation}
\end{lem}
\begin{proof}
Since $\Delta/2>1/2-k$, the contour in \eqref{int:I} is located to the right of all singularities
$s_j=1/2-k-j$ for $j=0,1,2,\ldots$
The function under the integral is bounded by $|\Im{s}|^{-1+\Re{s}}$.
Shifting the contour to the left, we have
\begin{multline*}
 I_{\pm}(0,v,k,x)=\sum_{j=0}^{\infty}(-1)^j\frac{\Gamma(k+j-it)}{\Gamma(2k+j)j!}
x^{1/2-k-j}e(\pm 1/4(1/2-k-j))=\\
=e(\pm 1/4(1/2-k))x^{1/2-k}\sum_{j=0}^{\infty}\frac{\Gamma(k+j-it)}{\Gamma(2k+j)j!}
\left(-\frac{e(\mp 1/4)}{x} \right)^j.
\end{multline*}
Thus
\begin{equation*}
 I_{\pm}(0,v,k,x)\ll x^{1/2-k}\sum_{m=0}^{\infty}\frac{|\Gamma(k+m-it)|}{m!|\Gamma(2k+m)|}\frac{1}{x^m}.
\end{equation*}
Using the functional equation for the Gamma function and the inequality
\begin{equation}\label{estimate1}
|k+j-it|\le k+j+|t|\le (k+j)(1+|t|) \quad \text{  for } j\geq0,
\end{equation}
we obtain
\begin{equation*}
|\Gamma(k+m-it)|=|\Gamma(k+it)|\prod_{j=0}^{m-1}|k+j-it|\le
|\Gamma(k+it)|T^m\frac{(k+m-1)!}{(k-1)!}.
\end{equation*}
Therefore,
\begin{equation*}\label{Ibound1}
 I_{\pm}(0,v,k,x)\ll x^{1/2-k}\frac{|\Gamma(k+it)|}{\Gamma(2k)}\sum_{m=0}^{\infty}\frac{T^m}{x^m}
\frac{(k+m-1)!(2k-1)!}{(k-1)!(2k+m-1)!m!}.
\end{equation*}
To estimate the factorials we rewrite them in the following way
\begin{equation*}\label{factorial1}
 \frac{(k+m-1)!(2k-1)!}{(k-1)!(2k+m-1)!m!}\ll
 \frac{(k+m)!}{(2k+2m)!} \frac{(2k+2m)!}{(2k+m)!m!}\frac{(2k)!}{k!} .
\end{equation*}
Applying Stirling's formula $n!\asymp\sqrt{n}(n/e)^n$
one has
\begin{multline*}\label{factorial2}
\frac{(2k+2m)!}{(2k+m)!m!}=\prod_{j=1}^m\frac{j+m}{j}\prod_{j=m+1}^{2k+m}\frac{j+m}{j}
\ll \\ \ll\frac{(2m)!}{m!m!}2^{2m}\ll\frac{1}{\sqrt{m+1}}2^{2k+2m},
\end{multline*}

\begin{equation*}\label{factorial3}
\frac{(k+m)!}{(2k+2m)!}\ll\frac{\sqrt{k+m}}{(k+m)!2^{2k+2m}}.
\end{equation*}
Using the last two estimates we obtain
\begin{equation*}\label{factorial4}
 \frac{(k+m-1)!(2k-1)!}{(k-1)!(2k+m-1)!m!}\ll
 \frac{\sqrt{k+m}}{\sqrt{m+1}}.
 \frac{(2k)!}{k!(k+m)!}.
\end{equation*}
This gives
\begin{equation*}\label{Ibound2}
 I_{\pm}(0,v,k,x)\ll x^{1/2-k}\frac{|\Gamma(k+it)|}{\Gamma(2k)}\frac{(2k)!}{k!}
 \sum_{m=0}^{\infty}\frac{T^m}{x^m} \frac{\sqrt{k+m}}{(k+m)!\sqrt{m+1}}.
\end{equation*}
Since
\begin{equation*}\label{factorial5}
 \frac{(k+m)!}{k!k^m}=\prod_{j=1}^m\frac{k+j}{k}>1
\end{equation*}
one has that for $\frac{T}{k}<x$
\begin{multline*}\label{Ibound3}
 I_{\pm}(0,v,k,x)\ll x^{1/2-k}\frac{|\Gamma(k+it)|}{\Gamma(2k)}\frac{(2k)!\sqrt{k}}{k!k!}
 \sum_{m=0}^{\infty}\left(\frac{T}{kx}\right)^m\ll \\
 \ll x^{1/2-k}\frac{|\Gamma(k+it)|}{\Gamma(2k)}2^{2k}.
\end{multline*}
 Applying Stirling's formula and \eqref{estimate1} we have for $\frac{T}{k}<x$
\begin{equation*}
 I_{\pm}(0,v,k,x)
 \ll e^{-\pi T/2}x^{1/2-k}T^{k-1/2}\frac{\Gamma(k)2^{2k}}{\Gamma(2k)}
\ll e^{-\pi T/2}\left(\frac{x}{T} \right)^{1/2}\left(\frac{eT}{xk} \right)^k.
\end{equation*}
\end{proof}

\subsection{The case of small $x$}
\begin{lem}\label{lem:Ismallx}
 For $x\leq Te/k$ we have
\begin{equation}
 e^{\pi T/2}I_{-}(0,v,k,x)\ll x^{-1/2}\max{(\frac{\sqrt{T}}{k}, \frac{1}{\sqrt{k}})},
\end{equation}
\begin{equation}
 e^{-\pi t/2}I_{+}(0,v,k,x)\ll x^{-1/2}\frac{1}{\sqrt{k+T}}.
\end{equation}
\end{lem}
\begin{proof}
Moving the contour of integration in \eqref{int:I} to $\Re{s}=-1/2$ we have
\begin{equation*}
 I_{\pm}(0,v,k,x)\ll x^{-1/2}\int_{-\infty}^{\infty}\frac{|\Gamma(k-1+ir)|}{|\Gamma(k+1-ir)|}
|\Gamma(1-i(r+t))|e^{\mp \pi r/2}dr.
\end{equation*}
Note that
\begin{equation*}
 \frac{|\Gamma(k-1+ir)|}{|\Gamma(k+1-ir)|}=\frac{1}{\sqrt{k^2+r^2}\sqrt{(k-1)^2+r^2}}
\ll \frac{1}{(k+|r|)^2}.
\end{equation*}
According to equation $8.332(3)$ of \cite{GR}
\begin{equation*}
 |\Gamma(1-i(r+t))|=\sqrt{\frac{\pi|r+t|}{\sinh{\pi|r+t|}}}
\ll e^{-\pi|r+t|/2}\left(|r+t|^{1/2}+1 \right).
\end{equation*}
Therefore,
\begin{equation*}
 e^{\mp\pi t/2}I_{\pm}(0,v,k,x)\ll x^{-1/2}\int_{-\infty}^{+\infty}\frac{1+|r+t|^{1/2}}{(k+|r|)^2}
e^{-\pi/2(|r+t|\pm (r+t))}dr.
\end{equation*}
Without loss of generality we assume that $t>1.$

First, we estimate $e^{-\pi t/2}I_{+}(0,v,k,x)$. Since
\begin{equation*}
 |r+t|+(r+t)=\begin{cases}
              2(r+t), & r \geq-t\\
              0,&r<-t
             \end{cases},
\end{equation*}
we have
\begin{multline*}
 x^{1/2}e^{-\pi t/2}I_{+}(0,v,k,x)\\ \ll
\int_{-\infty}^{-t}\frac{1+|r+t|^{1/2}}{(k+|r|)^2}dr+
\int_{-t}^{\infty}\frac{1+|r+t|^{1/2}}{(k+|r|)^2}e^{-\pi(r+t)}dr\\
\ll \int_{-\infty}^{-t-1}\frac{|r+t|^{1/2}}{(k+|r|)^2}dr+\frac{1}{(t+k)^2}
+\int_{-t+1}^{\infty}\frac{|r+t|^{1/2}}{(k+|r|)^2}e^{-\pi(r+t)}dr.
\end{multline*}
Consider the first integral
\begin{multline*}
 \int_{-\infty}^{-t-1}\frac{|r+t|^{1/2}}{(k+|r|)^2}dr=
\int_{1}^{\infty}\frac{r^{1/2}}{(r+k+t)^2}dr\\
\ll \int_{1}^{k+t}\frac{r^{1/2}}{(k+t)^2}dr+\int_{k+t}^{\infty}\frac{dr}{r^{3/2}}
\ll \frac{1}{\sqrt{k+t}}.
\end{multline*}
Let
\begin{equation*}
 f_1(r):=-\pi r+1/2\log{r}-2\log{(k+t-r)},
\end{equation*}
\begin{equation*}
 f_2(r):=-\pi r+1/2\log{r}-2\log{(k-t+r)}.
\end{equation*}
Since $f'_{1}(r),f'_{2}(r) <-1/2$  the second integral can be evaluated as follows
\begin{multline*}
 \int_{-t+1}^{\infty}\frac{|r+t|^{1/2}}{(k+|r|)^2}e^{-\pi(r+t)}dr=
\int_{1}^{\infty}\frac{r^{1/2}e^{-\pi r}}{(|r-t|+k)^2}dr\\
=\int_{1}^{t}\frac{r^{1/2}e^{-\pi r}}{(t-r+k)^2}dr+
\int_{t}^{\infty}\frac{r^{1/2}e^{-\pi r}}{(r-t+k)^2}dr
=\int_{1}^{t}e^{f_1(r)}dr+\int_{t}^{\infty}e^{f_2(r)}dr\\
\ll \int_{1}^{t}e^{f_1(r)}f_1'(r)dr+\int_{t}^{\infty}e^{f_2(r)}f_2'(r)dr
\ll\frac{1}{(k+t)^2}+\frac{\sqrt{t}e^{-\pi t}}{k^2}.
\end{multline*}
Finally, $$e^{-\pi t/2}I_{+}(0,v,k,x)\ll x^{-1/2}\frac{1}{\sqrt{k+t}}.$$

Similarly to the previous case
\begin{multline*}
 x^{1/2}e^{\pi t/2}I_{-}(0,v,k,x)\ll \int_{-\infty}^{-t-1}\frac{|r+t|^{1/2}}{(k+|r|)^2}e^{\pi(r+t)}dr+\\+
\int_{-t+1}^{\infty}\frac{|r+t|^{1/2}}{(|r|+k)^2}dr+\frac{1}{(k+t)^2}.
\end{multline*}
Estimating the first integral we have
\begin{equation*}
 \int_{-\infty}^{-t-1}\frac{|r+t|^{1/2}}{(k+|r|)^2}e^{\pi(r+t)}dr=\int_{1}^{\infty}
\frac{r^{1/2}}{(r+k+t)^2}e^{-\pi r}\ll \frac{1}{(k+t)^2}.
\end{equation*}
The second integral splits further into two parts
\begin{equation*}
 \int_{-t+1}^{\infty}\frac{|r+t|^{1/2}}{(|r|+k)^2}dr=
\int_{1}^{t}\frac{r^{1/2}}{(k+t-r)^2}dr+\int_{t}^{\infty}\frac{r^{1/2}}{(r+k-t)^2}dr.
\end{equation*}
Note that
\begin{equation*}
 \int_{t}^{\infty}\frac{r^{1/2}}{(r+k-t)^2}dr\ll
\begin{cases}
 t^{-1/2},& k<2t\\
k^{-1/2},&k>2t
\end{cases}
\end{equation*}
and
\begin{equation*}
 \int_{1}^{t}\frac{r^{1/2}}{(k+t-r)^2}dr=\frac{r^{1/2}}{k+t-r}\Bigl|_1^t-
 \int_{1}^{t}\frac{dr}{2(k+t-r)r^{1/2}}
 \ll \frac{\sqrt{t}}{k}.
\end{equation*}

Therefore,
$$e^{\pi t}I_{-}(0,v,k,x)\ll x^{-1/2}\max\left( \frac{\sqrt{t}}{k},\frac{1}{\sqrt{k}}\right).$$

\end{proof}

\section{Asymptotic formula for the first moment: $\nu \geq 3$}\label{section4}
Consider the twisted first moment of automorphic $L$-functions associated to primitive forms of  weight $2k\geq 4$ and level
$N=p^{\nu}$ with $p$ prime and $\nu \geq 3$
\begin{equation}
M_1(l,u,v)=\sum_{f \in H_{2k}^{*}(N)}^{h}\lambda_f(l)L_f(1/2+u+v).
\end{equation}

In this section we prove the following theorem
\begin{thm}
Let $k\geq 2 $, $p$ be a prime, $v=it, t \in \R$ and $T:=1+|t|$. For $N=p^{\nu},$ $\nu \geq 3$ we have
\begin{multline}\label{theorem nu3 statement}
 M_1(l,0,v)=\id_{(l,p)=1}\Biggl[\left( 1- \frac{1}{p}\right)\frac{1}{l^{1/2+v}} +\\+2 \pi i^{2k}V_N(0,v,k)
-\frac{2\pi i^{2k}}{p}V_{N/p}(0,v,k)\Biggr].
\end{multline}
Here
\begin{equation}\label{lemVestimate}
 V_N(0,v,k)\ll \begin{cases}
                 \frac{l^{1/2}(lT)^{\epsilon}}{N}\max{\left( \frac{\sqrt{T}}{k},\frac{1}{\sqrt{k}}\right)} & 4\pi elT\geq Nk,\\
                 \frac{1}{\sqrt{lT}}\left(2\pi e\frac{lT}{Nk}\right)^k & 4\pi elT<Nk.
                \end{cases}
\end{equation}
\end{thm}

\begin{lem}\label{lemM1=S-S}
Suppose that $\Re{u}>1/2$, $\Re{v}=0$, $k\geq 2$, $N=p^{\nu}$ and $\nu \geq 3$. Then
\begin{equation}
M_1(l,u,v)=\left(S(l,u,v;N)-\frac{1}{p}S(l,u,v;N/p)\right)\id_{(l,p)=1},
\end{equation}
\begin{equation}\label{eq:s}
 S(l,u,v;N)=\sum_{\substack{n=1\\(n,p)=1}}^{\infty}\frac{1}{n^{1/2+u+v}}\Delta_{2k,N}(l,n).
\end{equation}
\end{lem}
\begin{proof}
Consider
\begin{multline*}
M_1(l,u,v)=\sum_{n=1}^{\infty}\frac{1}{n^{1/2+u+v}}\Delta^{*}_{2k,N}(l,n)=\\
=\sum_{\substack{n=1\\(nl,p)=1}}^{\infty}\frac{1}{n^{1/2+u+v}}\left(\Delta_{2k,N}(l,n)-1/p\Delta_{2k,N/p}(l,n)\right)=\\=
\left(S(l,u,v;N)-\frac{1}{p}S(l,u,v;N/p)\right)\id_{(l,p)=1}.
\end{multline*}
\end{proof}
The condition $(n,p)=1$ can be removed by splitting $S(l,u,v;N)$ into two parts.
\begin{lem}\label{lemS=S1-S2}
One has
 \begin{equation}
   S(l,u,v;N)= S_1(l,u,v;N)-\frac{1}{p^{1/2+u+v}} S_2(l,u,v;N),
 \end{equation}
where
\begin{equation}\label{eq:s1}
S_1(l,u,v;N)= \sum_{n=1}^{\infty}\frac{1}{n^{1/2+u+v}}\Delta_{2k,N}(l,n)
\end{equation}
and
\begin{equation}\label{eq:s2}
S_2(l,u,v;N)= \sum_{n=1}^{\infty}\frac{1}{n^{1/2+u+v}}\Delta_{2k,N}(l,np).
\end{equation}
\end{lem}

\begin{lem}\label{lemS2=0}
Let $(l,p)=1$, $k\geq 2$, $N=p^{\nu}$ and $\nu \geq 3$. Then
\begin{equation}
 S_2(l,u,v;N)=S_2(l,u,v;N/p)=0.
\end{equation}

\end{lem}
\begin{proof}
 Using Petersson's trace formula \eqref{eq:Petersson}, we obtain
\begin{multline*}
 S_2(l,u,v;N)=\sum_{n=1}^{\infty}\frac{\delta(l,np)}{n^{1/2+u+v}}+2\pi i^{2k}
\sum_{n=1}^{\infty}\frac{1}{n^{1/2+u+v}}\times \\ \times \sum_{ q\equiv 0 \Mod{N}}\frac{Kl(l,np;q)}{q}J_{2k-1}
\left( \frac{4\pi \sqrt{lnp}}{q}\right).
\end{multline*}
Since $(l,p)=1$, we have $\delta(l,np)=0$.
Note that $N/p\geq p^2$ because $N=p^{\nu}$, $\nu \geq 3$. Thus, $q \equiv 0 \Mod{p^2}$ in the expressions for $S_2(l,u,v;N)$ and $S_2(l,u,v;N/p)$.
Further, one has $(l,p)=1$. Therefore, lemma \ref{Royer} implies $Kl(l,pn;q)=0$.
The assertion follows.
\end{proof}

\begin{lem}\label{lem:v}
Let $k\geq 2$, $N=p^{\nu}$ and $\nu \geq 3$. Let $\Re{u}>3/4$ and $\Re{v}=0$. Then
\begin{equation}\label{eq:s11}
   S_1(l,u,v;N)=\frac{1}{l^{1/2+u+v}}+2\pi i^{2k}V_N(u,v,k),
\end{equation}
where
\begin{equation}\label{eq:V}
 V_N(u,v,\lambda)=\sum_{n=1}^{\infty}\sum_{q\equiv 0 \Mod{N}}\frac{Kl(l,n;q)}{n^{1/2+u+v}q}J_{2\lambda-1}\left( 4\pi \frac{\sqrt{ln}}{q}\right).
\end{equation}
The series in \eqref{eq:V} over $n$ and $q$ converge absolutely for $\Re{\lambda}>3/4$, $\Re{u}>3/4$.
\end{lem}
\begin{proof}
Equation \eqref{eq:s11} follows from Petersson's trace formula \eqref{eq:Petersson}.
 Estimating the Bessel function and the Kloosterman sum, we have
\begin{multline*}
 V_N(u,v,\lambda)\ll \sum_{n=1}^{\infty}\frac{1}{n^{1/2+\Re{u}}}\sum_{q=1}^{\infty}\frac{1}{\sqrt{q}}
\min \left(\left(\frac{\sqrt{ln}}{q} \right)^{2\Re{\lambda}-1}, \left(\frac{q}{\sqrt{ln}} \right)^{1/2} \right)\\
=\sum_{n=1}^{\infty}\frac{1}{n^{1/2+\Re{u}}}\sum_{q<\sqrt{ln}}\frac{1}{(ln)^{1/4}}+
\sum_{n=1}^{\infty}\frac{1}{n^{1/2+\Re{u}}}\sum_{q>\sqrt{ln}}\frac{(ln)^{\lambda-1/2}}{q^{2\lambda-1/2}}.
\end{multline*}
Thus, the series over $q$ in the second sum is convergent for $2\lambda-1/2>1$.
Evaluating the sum over $q$ we have
\begin{equation*}
  V_N(u,v,\lambda)\ll_{l} \sum_{n=1}^{\infty}\frac{1}{n^{1/4+\Re{u}}}+\sum_{n=1}^{\infty}\frac{n^{\lambda-1/2}}{n^{1/2+\Re{u}}n^{\lambda-3/4}}.
\end{equation*}
Taking $\Re{u}>3/4$ the series over $n$ is absolutely convergent.
\end{proof}

Lemma \ref{lem:v} allows changing the order of summation
\begin{equation}
 V_{N}(u,v,\lambda)=\sum_{q \equiv 0\Mod{N}}\frac{1}{q}V_{N}(u,v,\lambda;q),
\end{equation}
\begin{equation}
 V_{N}(u,v,\lambda;q)=\sum_{n=1}^{\infty}\frac{1}{n^{1/2+u+v}}Kl(l,n;q)J_{2\lambda-1}\left( 4\pi \frac{\sqrt{ln}}{q}\right).
\end{equation}
For $\Re{s}>1-\Re{v}$ let
\begin{equation}
 g(s,v;q)=\sum_{n=1}^{\infty}\frac{Kl(l,n;q)}{n^{s+v}}.
\end{equation}
Note that
\begin{equation}
 g(s,v;1)=\zeta(s+v).
\end{equation}
The next step is to apply the Mellin-Barnes representation for the Bessel function.
\begin{lem}
For $\Re{u}>3/4$, $\Re{\lambda}>3/4$
 \begin{multline}\label{eq:vn}
  V_{N}(u,v,\lambda;q)=\frac{1}{4\pi i}\int_{(\Delta)}
\frac{\Gamma(\lambda-1/2+s/2)}{\Gamma(\lambda+1/2-s/2)}\times \\ \times g(1/2+u+s/2,v;q)\left(\frac{q}{2\pi\sqrt{l}} \right)^sds,
 \end{multline}
where $1-2\Re{\lambda}<\Delta<0$, $1-2\Re{u}<\Delta<0$.
\end{lem}

Using the properties of the Lerch zeta function we prove the functional equation and analytic continuation for $ g(s,v;q)$.
\begin{lem}\label{lem:g}
For $q \in \N$ and $ v \in \C$ the function $g(s,v;q)$ can be analytically continued on the whole complex plane as a function of complex variable $s$. Furthermore, for $\Re{(s+v)}<0$  one has
\begin{multline}
 g(s,v;q)=\Gamma(1-s-v)\left( \frac{2\pi}{q}\right)^{s+v-1}
\Biggl( e\left(\frac{1-s-v}{4} \right)  \times \\\times \sum_{\substack{n=1\\(n,q)=1}}^{\infty}\frac{e(n^{*}lq^{-1})}{n^{1-s-v}}+
e\left(-\frac{1-s-v}{4}\right) \sum_{\substack{n=1\\(n,q)=1}}^{\infty}\frac{e(-n^{*}lq^{-1})}{n^{1-s-v}}\Biggr),
\end{multline}
where $nn^{*}\equiv1\Mod{q}.$
\end{lem}
\begin{proof}
Consider
\begin{multline*}
 g(s,v;q)=\sum_{a,b=1}^{q}\delta_q(ab-1)e\left( \frac{al}{q}\right)\sum_{n=1}^{\infty}\frac{e(bn/q)}{n^{s+v}}=\\
=\sum_{a,b=1}^{q}\delta_q(ab-1)e\left(\frac{al}{q} \right)\zeta(0,b/q;s+v).
\end{multline*}
Applying the functional equation for the Lerch zeta function we have
\begin{multline*}
 g(s,v;q)
=\frac{\Gamma(1-s-v)}{(2\pi)^{1-s-v}}e\left( \frac{1-s-v}{4}\right)\times \\ \times \sum_{a,b=1}^{q}\delta_q(ab-1)e\left(\frac{al}{q} \right)\zeta(b/q,0;1-s-v)+
\frac{\Gamma(1-s-v)}{(2\pi)^{1-s-v}}\times \\ \times e\left(- \frac{1-s-v}{4}\right) \sum_{a,b=1}^{q}\delta_q(ab-1)e\left(\frac{al}{q} \right)\zeta(-b/q,0;1-s-v).
\end{multline*}

Note that
\begin{equation*}
\zeta(b/q,0;z)=\sum_{n+b/q>0}\frac{1}{(n+b/q)^z}=
q^z\sum_{\substack{n\equiv b \Mod{q}\\n>0}}\frac{1}{n^z}.
\end{equation*}
Therefore,
\begin{multline*}
 \sum_{a,b=1}^{q}\delta_q(ab-1)e\left(\frac{al}{q}\right)\zeta(b/q,0,1-s-v)=
q^{1-s-v}\sum_{n=1}^{\infty}\frac{1}{n^{1-s-v}}\times \\ \times \sum_{\substack{a,b=1\\b \equiv n \Mod{q}}}^{q}\delta_q(ab-1)e\left( \frac{al}{q}\right)
=q^{1-s-v}\sum_{\substack{n=1\\(n,q)=1}}^{\infty}\frac{e(n^{*}lq^{-1})}{n^{1-s-v}}.
\end{multline*}
Analogously,
\begin{multline*}
 \sum_{a,b=1}^{q}\delta_q(ab-1)e\left(\frac{al}{q}\right)\zeta(-b/q,0,1-s-v)=\\
=q^{1-s-v}\sum_{\substack{n=1\\(n,q)=1}}^{\infty}\frac{e(-n^{*}lq^{-1})}{n^{1-s-v}}.
\end{multline*}

\end{proof}

Now we can express $ V_N(u,v, \lambda)$ in terms of the integrals $I_{\pm}(u,v,\lambda,x)$ studied in section \ref{section3}.
\begin{lem}\label{lemV=}
For
$\Re \lambda-1>\Re u>3/4$  and  $\Re v=0$ one has
 \begin{multline}\label{eq:vnuv}
  V_N(u,v, \lambda)=\sum_{q\equiv 0 \Mod{N}}\frac{1}{q^{1/2+u+v}}
\sum_{\substack{n=1\\(n,q)=1}}^{\infty}\frac{e(n^{*}lq^{-1})}{n^{1/2-u-v}}(2\pi)^{u+v-1/2}\times \\ \times
e \left(\frac{1/2-u-v}{4}\right)I_-\left(u,v,\lambda,\frac{qn}{2\pi l}\right) +
\sum_{q\equiv 0 \Mod{N}}\frac{1}{q^{1/2+u+v}}\times \\ \times
\sum_{\substack{n=1\\(n,q)=1}}^{\infty}\frac{e(-n^{*}lq^{-1})}{n^{1/2-u-v}}(2\pi)^{u+v-1/2}
e \left(-\frac{1/2-u-v}{4}\right)I_+\left(u,v,\lambda,\frac{qn}{2\pi l}\right),
 \end{multline}
where  $1-2\Re{\lambda}<\Delta<-1-2\Re{u}$ and
\begin{multline}\label{Idef}
 I_{\pm}(u,v,\lambda,x)=\frac{1}{4\pi i}\int_{(\Delta)} \frac{\Gamma(\lambda-1/2+s/2)}{\Gamma(\lambda+1/2-s/2)}\times \\ \times\Gamma(1/2-u-v-s/2)
\left(x \right)^{s/2}e\left( \pm \frac{s}{8}\right)ds.
\end{multline}

\end{lem}
\begin{proof}
 Moving the contour of integration in \eqref{eq:vn} to $1-2\Re{\lambda}<\Delta<-1-2\Re{u}$
\begin{multline*}
 V_N(u,v, \lambda)=\frac{1}{4\pi i}\int_{(\Delta)}\frac{\Gamma(\lambda-1/2+s/2)}{\Gamma(\lambda+1/2-s/2)}\times \\ \times \sum_{q\equiv 0 \Mod{N}}
g(1/2+u+s/2,v;q)\frac{1}{q}\left(\frac{q}{\pi \sqrt{l}} \right)^{s}ds.
\end{multline*}
By Lemma \ref{lem:g} we have
\begin{multline*}
 V_N(u,v, \lambda)=\frac{1}{4\pi i}\int_{(\Delta)}\frac{\Gamma(\lambda-1/2+s/2)}{\Gamma(\lambda+1/2-s/2)}\Gamma(1/2-u-v-s/2)\times \\ \times
\sum_{q \equiv 0\Mod{N}}\frac{1}{q}\left( \frac{2\pi}{q}\right)^{s/2+u+v-1/2}
\left( \frac{q}{2\pi \sqrt{l}}\right)^s\times\\ \times
\left( e\left( \frac{1/2-u-v-s/2}{4}\right)\sum_{\substack{n=1\\(n,q)=1}}^{\infty}\frac{e(n^{*}lq^{-1})}{n^{1/2-u-s/2-v}}+\right.\\ \left.
+ e\left(- \frac{1/2-u-v-s/2}{4}\right)\sum_{\substack{n=1\\(n,q)=1}}^{\infty}\frac{e(-n^{*}lq^{-1})}{n^{1/2-u-s/2-v}}\right) ds.
\end{multline*}
Note that for $-u+\Delta/2<0$ the integral over $s$ converges since
\begin{equation*}
 \frac{\Gamma(\lambda-1/2+s/2)}{\Gamma(\lambda+1/2-s/2)}\Gamma(1/2-u-v-s/2)\asymp \Re{s}^{-1-u+\Delta/2}e^{-\pi|\Im{s}|/4}.
\end{equation*}
If $1/2-u-\Delta/2>1$, then the series over $n$ and $q$ are convergent as well and, therefore, we can change the order of summation and integration.
The assertion follows.
\end{proof}
Finally,  explicit formula \eqref{theorem nu3 statement} follows from lemmas \ref{lemM1=S-S}, \ref{lemS=S1-S2}, \ref{lemS2=0}, \ref{lem:v} and \ref{lemV=}
by means of analytic continuation to  $u=0,\lambda=k\geq2.$  We are left to prove estimate \eqref{lemVestimate}.
\begin{lem}\label{lem:estv}
Let $k\geq 2 $, $t \in \R$ and $T:=1+|t|$. Let $d_0=4\pi e\frac{lT}{Nk}$. Then
\begin{equation}
 V_N(0,it,k)\ll \begin{cases}
                 \frac{l^{1/2}}{N}d_{0}^{\epsilon}\max{\left( \frac{\sqrt{T}}{k},\frac{1}{\sqrt{k}}\right)}, & d_0\geq 1\\
                 \frac{1}{\sqrt{lT}}\left( \frac{d_0}{2}\right)^k, & d_0<1.
                \end{cases}
\end{equation}

\end{lem}
\begin{proof}
 Estimating \eqref{eq:vnuv} one has
\begin{equation*}
 V_N(0,v, k)\ll \frac{1}{N^{1/2}}\sum_{d=1}^{\infty}\frac{d^{\epsilon}}{d^{1/2}}
e^{\pm \pi t/2}I_{\pm}\left(0,v,k,\frac{dN}{2\pi l}\right).
\end{equation*}
We apply lemmas \ref{lem:Ilargex} and \ref{lem:Ismallx} to bound the last expression.
For $d_0\geq 1$
\begin{multline*}
 V_N(0,v, k)\ll \frac{1}{\sqrt{N}}\sum_{d\leq d_0}\frac{d^{\epsilon}}{d^{1/2}}
\left(\frac{l}{dN} \right)^{1/2}\max{\left( \frac{\sqrt{T}}{k},\frac{1}{\sqrt{k}}\right)}+\\
+\frac{1}{\sqrt{N}}\sum_{d> d_0}\frac{d^{\epsilon}}{d^{1/2}}
\left(\frac{dN}{lT} \right)^{1/2}\left( \frac{2\pi e lT}{dNk}\right)^k\ll  \\
\ll \frac{\sqrt{l}}{N}d_{0}^{\epsilon}\max{\left( \frac{\sqrt{T}}{k},\frac{1}{\sqrt{k}}\right)}+
\frac{d_{0}^{\epsilon}\sqrt{lT}}{Nk2^k}.
\end{multline*}
For $d_0<1$
\begin{multline*}
 V_N(0,v, k)\ll \frac{1}{\sqrt{N}}\sum_{d=1}^{\infty}
\frac{d^{\epsilon}}{d^{1/2}}\left( \frac{dN}{lT}\right)^{1/2}
\left( \frac{2\pi elT}{dNk}\right)^k
\ll \frac{1}{\sqrt{lT}}\left( \frac{d_0}{2}\right)^k.
\end{multline*}
\end{proof}

\section{Asymptotic formula for the first moment: $\nu=2$}\label{section5}
\begin{thm}
Suppose that $v=it, t \in \R$, $k \geq 2$ and $N=p^2$ with $p$ prime. Then
 \begin{multline}
M_1(l,0,v)=\id_{(l,p)=1}\Biggl[\left(1-\frac{1}{p-p^{-1}}\right)\frac{1}{l^{1/2+v}}-\\-\frac{(2\pi)^{2v}i^{2k}}{p-p^{-1}} \frac{\Gamma(k-v)}{\Gamma(k+v)}\frac{1}{p^{1+2v}l^{1/2-v}}
+O\left( V_{p^2}(0,v,k)+\frac{1}{p}V_p(0,v,k)\right)\Biggr].
 \end{multline}
where $V_{N}(0,v,k)$ can be  estimated by \eqref{lemVestimate}.
\end{thm}
Similarly to lemma \ref{lemM1=S-S} we prove the following decomposition.

\begin{lem}
Suppose that $\Re{u}>1/2$, $\Re{v}=0$, $k\geq 2$, $N=p^{2}$. Then
\begin{equation}
M_1(l,u,v)=\left(S(l,u,v;p^2)-\frac{1}{p-p^{-1}}S(l,u,v;p)\right)\id_{(l,p)=1},
\end{equation}
where $S(l,u,v;N)$ is given by \eqref{eq:s}.
\end{lem}
The sum $ S(l,u,v;p^2)$ can be evaluated with the methods of the previous section. This gives
\begin{equation}
 S(l,0,v;p^2)=S_1(l,0,v;p^2)=\frac{1}{l^{1/2+v}}+2\pi i^{2k}V_{p^2}(0,v,k),
\end{equation}
where $S_1(l,u,v;p^2)$ is defined by \eqref{eq:s1}.
The sum  $ S(l,u,v;p)$ splits into two parts
\begin{equation}\label{sumsp}
 S(l,u,v;p)=S_1(l,u,v;p)-\frac{1}{p^{1/2+u+v}}S_2(l,u,v;p).
\end{equation}
Here
\begin{equation}\label{sumsp1}
 S_1(l,0,v;p)=\frac{1}{l^{1/2+v}}+2\pi i^{2k}V_p(0,v,k)
\end{equation}
and $S_2(l,u,v;p)$ is given by \eqref{eq:s2}.
Note that if $N=p$ we cannot use the property of vanishing of Kloosterman sums
\eqref{Royer} to evaluate $S_2(l,u,v;p)$.

\begin{lem}
 For $(l,p)=1$ we have
\begin{equation}\label{sumsp2}
 S_2(l,u,v;p)=2\pi i^{2k}\widetilde{V}_{p}(u,v,k),
\end{equation}
where
\begin{equation}
 \widetilde{V}_{p}(u,v,k)=\sum_{n=1}^{\infty}\sum_{q\equiv 0 \Mod{p}}
\frac{Kl(l,np;q)}{n^{1/2+u+v}q}J_{2k-1}\left( 4\pi \frac{\sqrt{lnp}}{q}\right).
\end{equation}

\end{lem}
In the region of absolute convergence we can change the order of summations.
\begin{lem}
 For $\Re{u}>3/4$, $\Re{\lambda}>3/4$ we have
\begin{equation}
 \widetilde{V}_{p}(u,v,\lambda)=\sum_{q\equiv 0\Mod{p}}\frac{1}{q}\widetilde{V}_p(u,v,\lambda;q),
\end{equation}
\begin{equation}
 \widetilde{V}_p(u,v,\lambda;q)=\sum_{n=1}^{\infty}\frac{1}{n^{1/2+u+v}}
Kl(l,np;q)J_{2\lambda-1}\left( 4\pi \frac{\sqrt{lnp}}{q}\right).
\end{equation}
\end{lem}
For $\Re{s}>1-\Re{v}$ let
\begin{equation}
 \widetilde{g}(s,v;q)=\sum_{n=1}^{\infty}\frac{Kl(l,np;q)}{n^{s+v}}.
\end{equation}
Using the definition of Kloosterman sums
\begin{equation}\label{eq:gtilde}
 \widetilde{g}(s,v;q)=\sum_{a,b=1}^{q}\delta_q(ab-1)e\left( \frac{al}{q}\right)
\zeta(0,bp/q;s+v).
\end{equation}
Since $(b,qp)=1$ the function $\zeta(0,bp/q,s+v)$ has a pole only if $q=1$ or $q=p$. Because $(l,p)=1$ we have
\begin{equation*}
 Kl(l,np;p)=\sum_{a,b=1}^{p}\delta_p(ab-1)e\left( \frac{al}{p}\right)=
\sum_{a=1}^{p-1}e\left(\frac{al}{p} \right)=-1.
\end{equation*}
This gives
\begin{equation}\label{eq:gtilde p}
 \widetilde{g}(s,v;p)=-\zeta(s+v).
\end{equation}



Applying the Mellin-Barnes representation for the Bessel function we obtain
\begin{lem}
For $\Re{u}>3/4$, $\Re{\lambda}>3/4$
\begin{multline}
 \widetilde{V}_p(u,v,\lambda;q)=\frac{1}{4\pi i}\int_{(\Delta)}
\frac{\Gamma(\lambda-1/2+s/2)}{\Gamma(\lambda+1/2-s/2)}\times \\ \times
\widetilde{g}(1/2+u+s/2,v;q)\left(\frac{q}{2\pi \sqrt{lp}} \right)^sds,
\end{multline}
where $ \max{(1-2\Re{\lambda},1-2\Re{u})}<\Delta<0.$
\end{lem}


\begin{lem}
For $q \in \N$,  $p|q$ and $ v \in \C$ the function $ \widetilde{g}(s,v;q)$ can be meromorphically continued on the whole complex plane as a function of complex variable $s$. Furthermore, for $\Re{(s+v)}<0$  one has
\begin{multline}
 \widetilde{g}(s,v;q)=\Gamma(1-s-v)\left( \frac{2\pi p}{q}\right)^{s+v-1}\id_{(q/p,p)=1}
\times \\ \times \Biggl( e\left(\frac{1-s-v}{4}\right)\sum_{n=1}^{\infty}\frac{f_+(n)}{n^{1-s-v}}+
e\left(-\frac{1-s-v}{4}\right)\sum_{n=1}^{\infty}\frac{f_-(n)}{n^{1-s-v}}  \Biggr),
\end{multline}
where the functions
\begin{equation}\label{eq:fpm}
 f_{\pm}(n)=\id_{(n,q/p)=1}e\left( \pm\frac{ln^{*}}{q}\right)
\sum_{\substack{s=1\\(\pm n^{*}+sq/p,p)=1}}^{p}e\left(\frac{ls}{p} \right)
\end{equation}
with $nn^*\equiv 1 \Mod{q/p}$ satisfy
\begin{equation}\label{fineq}
 |f_{\pm}(n)|\leq 1.
\end{equation}

\end{lem}
\begin{proof}
Applying the functional equation for the Lerch zeta function  we have
\begin{multline*}
  \widetilde{g}(s,v;q)
=\frac{\Gamma(1-s-v)}{(2\pi)^{1-s-v}}e\left( \frac{1-s-v}{4}\right)\times \\ \times \sum_{a,b=1}^{q}\delta_q(ab-1)e\left(\frac{al}{q} \right)\zeta(bp/q,0;1-s-v)+
\frac{\Gamma(1-s-v)}{(2\pi)^{1-s-v}}\times \\ \times e\left(- \frac{1-s-v}{4}\right) \sum_{a,b=1}^{q}\delta_q(ab-1)e\left(\frac{al}{q} \right)\zeta(-bp/q,0;1-s-v).
\end{multline*}
Note that
\begin{multline*}
 \sum_{a,b=1}^{q}\delta_q(ab-1)e\left(\frac{al}{q} \right)\zeta(bp/q,0;1-s-v)=\\=
\left( \frac{q}{p}\right)^{1-s-v}\sum_{n=1}^{\infty}\frac{1}{n^{1-s-v}}
\sum_{\substack{a,b=1\\b \equiv n \Mod{q/p}}}^{q}\delta_q(ab-1)e\left( \frac{al}{q}\right).
\end{multline*}
Consider the sum over $a,b$. Since $b \equiv n\Mod{q/p}$, there is $t \Mod{p}$ such that
$b=n+tq/p$. Then $ab\equiv 1\Mod{q}$ implies $an+atq/p\equiv 1\Mod{q}$. Further,
\begin{equation*}
 \begin{cases}
  an \equiv 1 \Mod{q/p}\\
  at+p(an-1)/q\equiv 0\Mod{p}\\
  (a,p)=1
 \end{cases}
\Leftrightarrow
\begin{cases}
 an\equiv 1\Mod{q/p}\\
(a,p)=1
\end{cases}.
\end{equation*}
Therefore,
\begin{equation*}
S:= \sum_{\substack{a,b=1\\b \equiv n \Mod{q/p}}}^{q}\delta_q(ab-1)e\left( \frac{al}{q}\right)=
\sum_{\substack{a=1\\(a,p)=1}}^{q}\delta_{q/p}(an-1)e\left(\frac{al}{q} \right).
\end{equation*}
Condition $an\equiv 1\Mod{q/p}$ implies that
\begin{equation*}
 \begin{cases}
  (n,q/p)=1\\
  a\equiv n^{*}\Mod{q/p}
 \end{cases}
\Leftrightarrow
\begin{cases}
 (n,q/p)=1\\
  a= n^{*}+sq/p\text{, }s\Mod{p}
\end{cases}.
\end{equation*}
Consequently,
\begin{equation*}
 S=
\id_{(n,q/p)=1}
e\left( \frac{ln^{*}}{q}\right)
\sum_{\substack{s=1\\(n^{*}+sq/p,p)=1}}^{p}e\left(\frac{ls}{p} \right).
\end{equation*}
The inequality \eqref{fineq} can be proved as follows.
If $(q/p,p)=1$ then requirement $(n^{*}+sq/p,p)=1$ is not satisfied only for one $s$. Since
$(l,p)=1$ we conclude that
\begin{equation*}
 \sum_{s=1}^{p}e\left( \frac{ls}{p}\right)=0,
\end{equation*}
and, therefore,
\begin{equation*}
 \left|\sum_{\substack{s=1\\(n^{*}+sq/p,p)=1}}^{p}e \left( \frac{ls}{p}\right)\right|\leq 1.
\end{equation*}
If $q/p\equiv 0\Mod{p}$ then condition $(n^{*}+sq/p,p)=1$ is equivalent to $(n^{*},p)=1$.
In this case
\begin{equation*}
\sum_{\substack{s=1\\(n^{*}+sq/p,p)=1}}^{p}e \left( \frac{ls}{p}\right)=0.
\end{equation*}

\end{proof}

The following statement can be proved in the same manner as lemma \ref{lemV=}
\begin{lem}
For
$\Re \lambda-1>\Re u>3/4$  and  $\Re v=0$ one has
\begin{multline}\label{eq:vtilde2}
 \widetilde{V}_p(u,v,\lambda)=-\frac{1}{p}\frac{\Gamma(\lambda-u-v)}{\Gamma(\lambda+u+v)}
\left( \frac{\sqrt{p}}{2\pi \sqrt{l}}\right)^{1-2u-2v}\\
+
\sum_{\substack{q \equiv 0 \Mod{p}\\ (q/p,p)=1}}\frac{(2\pi p)^{u+v-1/2}}{q^{1/2+u+v}} \\
\times \Biggl(\sum_{n=1}^{\infty}\frac{f_+(n)}{n^{1/2-u-v}}
e\left(\frac{1/2-u-v}{4}\right) I_{-}\left(u,v,\lambda,\frac{qn}{2\pi l} \right)\\+
\sum_{n=1}^{\infty}\frac{f_-(n)}{n^{1/2-u-v}}
e\left(-\frac{1/2-u-v}{4}\right)I_{+}\left(u,v,\lambda,\frac{qn}{2\pi l} \right)
\Biggr),
\end{multline}
where $I_{\pm}(u,v,\lambda,x)$ and $f_{\pm}(n)$ are defined by \eqref{Idef} and \eqref{eq:fpm}, respectively.
\end{lem}

\begin{cor}
Let $N=p^2$. Then
\begin{equation}\label{eqv3}
 \widetilde{V}_p(0,v,k)=-\frac{1}{p}\frac{\Gamma(k-v)}{\Gamma(k+v)}
\left(\frac{\sqrt{p}}{2\pi\sqrt{l}}\right)^{1-2v}+O\left( \frac{1}{\sqrt{p}}V_p(0,v,k)\right),
\end{equation}

\begin{equation}\label{eq:slp}
 S(l,0,v;p)=\frac{1}{l^{1/2+v}}+\frac{2\pi i^{2k}}{(2\pi)^{1-2v}}\frac{\Gamma(k-v)}{\Gamma(k+v)}
\frac{l^{v-1/2}}{p^{1+2v}}+O(V_p(0,v,k)).
\end{equation}

\end{cor}
\begin{proof}
First, we make analytic continuation to  $u=0,\lambda=k\geq2.$
The inequality \eqref{fineq} allows  estimating the series over $n$ in \eqref{eq:vtilde2}
similarly to lemma \ref{lem:estv}. This proves \eqref{eqv3}.
Equation \eqref{eq:slp} follows from \eqref{sumsp}, \eqref{sumsp1}, \eqref{sumsp2} and \eqref{eqv3}.
\end{proof}

\section{Mollification at the center of the critical strip}\label{section6}
Using the technique of mollification and theorems \ref{thm:main2}, \ref{thm:Bet}, \ref{thm:BF}  we prove a lower bound on the proportion of non-vanishing $L$-functions at the critical point $1/2$. Note that due to the uniformity of our results in $t$-aspect a positive proportion of non-vanishing $L$-values can also be established at any point on the critical line $1/2+it$ such that $|t|<N$.

We choose the standard mollifier of the form
\begin{equation}
 X(f)=\sum_{l \leq L}\frac{x_l\lambda_f(l)}{\sqrt{l}},
\end{equation}
where $(x_l)$ are real coefficients and $L=N^{\Omega}$ is called the length of mollifier.
Then
\begin{equation}
M_{1}^h=\sum_{f \in H_{2k}^{*}(N)}^{h}L_f(1/2)X(f)=\sum_{l \leq L}\frac{x_l}{\sqrt{l}}M_1(l,0,0)
\end{equation}
and
\begin{equation}
M_{2}^h=\sum_{f \in H_{2k}^{*}(N)}^{h}L^{2}_{f}(1/2)X^2(f)=\sum_{\substack{d \leq L\\ l_1d \leq L\\ l_2d \leq L}}\frac{x_{dl_1}x_{dl_2}}{d\sqrt{l_1l_2}}M_2(l_1l_2,0,0).
\end{equation}
Theorems  \ref{thm:main2}, \ref{thm:Bet}, \ref{thm:BF}  imply that any $\Omega<1$ is admissible for the first moment and $\Omega<1/2$ for the second moment. Therefore, the length of mollifier can be taken up to $N^{1/2}$.
Following \cite{R2} we set
\begin{equation}
x_l=\sum_{lm\leq N^{\Omega}}\frac{\mu*\mu(m)}{m}y_{lm},
\end{equation}
where
\begin{equation}
y_n=\begin{cases}
n\mu(n)/\phi(n)&  \text{ if }n\leq N^{\Omega} \text{ and }p\not{|}n,\\
0 & \text{ otherwise.}
\end{cases}
\end{equation}

By Cauchy-Schwartz inequality (see \cite{R2} for details)
\begin{equation}
 \sum_{\substack{f \in H_{2k}^{*}(N)\\ L_f(1/2)\neq 0}}1 \geq \frac{(M_{1}^h)^2}{M_{2}^{h}}
\geq \frac{p-1}{p}\frac{\Omega}{1+2\Omega}-\epsilon.
\end{equation}
Taking $\Omega=1/2-\epsilon$ as $N\rightarrow \infty$ we prove theorem \ref{thm:main}.
\section*{Funding}
The work of Olga Balkanova was supported by the Russian Science Foundation under grant $14-11-00335$ and performed at the Institute for Applied Mathematics of Russian Academy of Sciences.

The work of Dmitry Frolenkov  was supported by the Russian Science Foundation under grant 14--50--00005 and performed in Steklov Mathematical Institute of Russian Academy of Sciences.

\section*{Acknowledgments}
The authors thank Sandro Bettin for extending his result \cite{Bet} to the case of prime powers  for our application.

\nocite{}

\end{document}